\documentclass[12pt,a4paper,reqno]{amsart}

\usepackage[a4paper,margin=23mm,bottom=23mm,top=26mm]{geometry}
\usepackage{amsmath,amssymb,amsthm,amscd}
\usepackage{mathtools}
\usepackage[T1]{fontenc}
\usepackage{newtxtext,newtxmath}
\usepackage[colorlinks=true,linkcolor=cyan,citecolor=magenta,urlcolor=blue]{hyperref}
\usepackage[all,ps]{xypic}
\usepackage{tikz,tikz-cd}
\usepackage{enumerate}

\newtheorem{theorem}{Theorem}[section]
\newtheorem{proposition}[theorem]{Proposition}
\newtheorem{lemma}[theorem]{Lemma}

\theoremstyle{definition}
\newtheorem{definition}[theorem]{Definition}
\newtheorem{remark}[theorem]{Remark}

\newcommand{\Hom}{\mathrm{Hom}}

\newcommand{\Tor}{\mathrm{Tor}}

\newcommand{\Vect}{\mathbf{Vect}}

\newcommand{\OrdFor}{\mathbf{OrdFor}}
\newcommand{\Ch}{\mathbf{Ch}}

\newcommand{\ho}{\mathrm{ho}}
\newcommand{\tr}{\tau}
\newcommand{\op}{\mathrm{op}}
\newcommand{\id}{\mathrm{id}}
\newcommand{\inj}{\mathrm{inj}}
\newcommand{\sur}{\mathrm{sur}}

\newcommand{\lmod}[1]{{#1}\text{-}\mathbf{Mod}}
\newcommand{\rmod}[1]{\mathbf{Mod}\text{-}{#1}}

\numberwithin{equation}{section}

\title{Combinatorial Models for Linear Homotopy Theories}
\author{Atabey Kaygun}
\address{Istanbul Technical University, Istanbul}
\email{kaygun@itu.edu.tr}

\begin{document}
\maketitle

\begin{abstract}
  For a field $k$ of characteristic $0$, we compare $k$-linear chain complexes, semisimplicial
  vector spaces, augmented semisimplicial vector spaces, semicubical vector spaces, and arboreal
  vector spaces through small differential categorical algebras.  We prove that semisimplicial
  modules and augmented semisimplicial modules are equivalent to appropriate chain-complex homotopy
  theories, both at the Gabriel--Zisman localization and the Quillen model-categorical level.  The
  semicubical sign embedding gives a natural comparison from semicubical modules to augmented
  semisimplicial modules and induces a Quillen adjunction, but not a Quillen equivalence on the
  full semicubical category since there is an obstruction in augmented homology at degree $-1$.
\end{abstract}

\section*{Introduction}

We compare several $k$-linear combinatorial models for $k$-linear homotopy theories.  The main
models are chain complexes, semisimplicial vector spaces, augmented semisimplicial vector spaces,
semicubical vector spaces, and arboreal vector spaces.  The comparison has two different
characters.  Chain complexes, semisimplicial modules, augmented semisimplicial modules, and
arboreal modules give equivalent homotopy theories in the appropriate degrees.  The semicubical
model, by contrast, admits a natural comparison with the augmented semisimplicial model, but the
comparison is not an equivalence on the full semicubical category.

The present paper separates these two phenomena.  For semisimplicial modules, restriction along
the differential categorical algebra $\Omega$ identifies the relevant homology with the homology
of a nonnegatively graded chain complex.  This gives both a Gabriel--Zisman equivalence with
$\Ch_{\geq 0}(k)$ and a Quillen equivalence after transferring the standard model structure.
For augmented semisimplicial modules, the analogous algebra is $\Omega_a$, obtained from
$\Omega$ by adjoining the object $[-1]$ and the differential $d_0\colon [-1]\to [0]$.  The
full augmented chain functor identifies the localized homotopy theory with
$\Ch_{\geq -1}(k)$, again both after Gabriel--Zisman localization and at the Quillen level.

The semicubical theory is different.  We use the sign embedding
$v\colon k[\Delta_{a,\inj}]\to k[\square_{\inj}]$, defined on cofaces by
$v(\delta^i)=\delta^1_{i+1}-\delta^0_{i+1}$.  Restriction along $v$ turns a semicubical module
into an augmented semisimplicial module whose augmented chain complex is the shifted semicubical
sign complex.  Thus the functor $v^\ast$ detects the semicubical weak equivalences.  However,
the induction--restriction adjunction $v_!\dashv v^\ast$ is not an equivalence after
localization: the unit can fail to preserve the degree $-1$ augmented homology.  Equivalently,
the semicubical sign model maps naturally to the augmented semisimplicial homotopy theory, but it
does not present the same homotopy theory on the full category.

The arboreal part of the comparison is imported from \cite{Kaygun2026}.  In that work, the
category $\OrdFor$ of ordered forests is shown to compare with the augmented semisimplicial
model through the height-one shadow functor, using the shifted identification
$\Delta_{\mathrm{epi}}^{\op}\cong \Delta_{a,\inj}$.  The role of the present paper is not to
reprove those arboreal results, but to place them in the same diagram as the semisimplicial,
augmented semisimplicial, chain-complex, and semicubical comparisons.  The resulting picture has
two genuine equivalence lines, namely the chain--semisimplicial and chain--augmented
semisimplicial lines, together with an arboreal equivalence imported from \cite{Kaygun2026} and a
semicubical comparison functor which is not an equivalence.

The simplicial and augmented simplicial background is classical.  The ordinary Dold--Kan
correspondence is standard; see \cite[Chapter~III, \S2]{WeibelHA} and
\cite[Chapter~II, \S5]{Quillen1967HomotopicalAlgebra}.  Augmented homology appears already in
\cite{Fors2001AugmentalHomology}.  Dold--Kan type equivalences have also been placed in a broader
combinatorial framework by Lack and Street \cite{LackStreet2015}, and the classical comparison has
recently been reinterpreted in crossed simplicial terms by Kaya and Kaygun
\cite{KayaKaygun2024}.  A modern occurrence of the intermediate semisimplicial category
$\Delta_{\inj}$ appears in work of Brantner, Hahn, and Knudsen
\cite{BrantnerHahnKnudsen2024}, where the classical Quillen equivalence is factored through
semisimplicial modules.  The semisimplicial comparison used here is the injective form of the
$\Omega$-semisimplicial comparison appearing in \cite{KayaKaygun2024}.

The model-categorical input is supplied by induced model structures.  The chain-complex categories
$\Ch_{\geq 0}(k)$ and $\Ch_{\geq -1}(k)$ carry the standard model structures in which weak
equivalences are quasi-isomorphisms, cofibrations are degreewise monomorphisms, and fibrations are
degreewise epimorphisms.  We transfer these model structures along the comparison functors from
$\Omega$ and $\Omega_a$, using freeness over the corresponding differential categorical
algebras.  This produces Quillen equivalences for the semisimplicial and augmented
semisimplicial models.  We also transfer the augmented semisimplicial model structure along the
semicubical sign embedding.  The transferred semicubical model structure has the expected weak
equivalences, but the resulting Quillen adjunction is not a Quillen equivalence.  The general
technology of induced model structures is standard; see
Hess--K\k{e}dziorek--Riehl--Shipley \cite{HessKedziorekRiehlShipley2017}.  The criterion used
below is stated in the form needed for our $k$-linear diagrammatic setting.

Semisimplicial and semicubical objects have long served as alternatives to simplicial objects.
Semisimplicial objects are indexed by the injective simplex category.  In the augmented setting,
the shifted duality $\Delta_{\mathrm{epi}}^{\op}\cong \Delta_{a,\inj}$ makes the surjective
simplex category equally natural, especially for the arboreal height-one shadow.  On the cubical
side, Brown and Higgins studied the algebra of cubes in \cite{BrownHiggins1981AlgebraOfCubes} and
proved that cubical abelian groups with connections are equivalent to chain complexes in
\cite{BrownHiggins2003CubicalAbelian}.  Grandis and Mauri developed the cubical site with
connections and interchanges in \cite{GrandisMauri2003}, while Lack and Street placed the cubical
category of Crans and Verity inside a general theory of Dold--Kan type combinatorial equivalences
\cite{LackStreet2015}.  The semicubical theory studied here is deliberately smaller: it has no
degeneracies, no connections, and no interchanges.  Its comparison with augmented semisimplicial
objects is governed only by the signed difference of the two cubical coface families.

The tree-shaped background points in a related but distinct direction.  Trees have been central to
operad theory since Boardman and Vogt \cite{BoardmanVogt1973} and remain the standard
combinatorics for free nonsymmetric operads; see \cite[Chapter~5]{MarklShniderStasheff2002} and
\cite[Chapter~4]{Leinster2004}.  Moerdijk and Weiss introduced dendroidal sets by replacing
$\Delta$ with a category of trees in order to model operads and homotopy operads
\cite{MoerdijkWeiss2007}, and the resulting homotopy theory was developed further by Cisinski and
Moerdijk \cite{CisinskiMoerdijk2011}; see also
\cite{Weiss2011FromOperadsToDendroidalSets}.  For planar trees, Guti\'errez, Luk\'acs, and Weiss
proved a Dold--Kan type correspondence for dendroidal abelian groups
\cite{GutierrezLukacsWeiss2011}.  The category $\OrdFor$ used in \cite{Kaygun2026} is not a
dendroidal category in this sense.  Its objects are finite ordinals, its morphisms are ordered
forests, and its quotient to $\Delta_{\mathrm{epi}}^{\op}$ remembers only the height-one linear
shadow.  Thus the arboreal comparison used here is closer to a planar many-output extension of
semisimplicial combinatorics than to a replacement for the dendroidal encoding of operadic
substitution.

\subsection*{Structure of the paper}

Section~\ref{sec:indexing-categories} introduces the indexing categories and the functors relating
the semisimplicial, augmented semisimplicial, semicubical, and arboreal settings.
Section~\ref{sec:categorical-algebra} passes to the $k$-linear setting: it introduces the
differential categorical algebras $\Omega$ and $\Omega_a$, defines the comparison functors,
and proves the freeness statements needed for change of rings.  Section~\ref{sec:homological-algebra}
defines the weak equivalences.  The nonaugmented semisimplicial and semicubical theories are
controlled by nonnegative chain homology, while the augmented semisimplicial theory is controlled
by the full augmented chain complex, equivalently by good truncation together with the degree
$-1$ invariant.  Section~\ref{sec:gz-localization} proves the corresponding
Gabriel--Zisman localization statements.  The semisimplicial and augmented semisimplicial
comparisons are equivalences; the semicubical sign-shadow comparison is a functor to the
augmented semisimplicial localization but is not an equivalence on the full semicubical category.
Finally, Section~\ref{sec:quillen-equivalences} constructs the corresponding right-induced model
structures.  It proves the semisimplicial and augmented semisimplicial Quillen equivalences,
constructs the semicubical Quillen adjunction, and incorporates the arboreal Quillen equivalence
from \cite{Kaygun2026} into the final comparison diagram.

\section{Indexing categories}\label{sec:indexing-categories}

We follow May~\cite[Chapter~1]{May1967} for the simplex category, Ebert--Randal-Williams~\cite[\S1.1]{EbertRandalWilliams2019}
for the augmented and semisimplicial variants, and Grandis--Mauri~\cite[\S1]{GrandisMauri2003} for the cubical category.

For each integer $n\geq 0$ write $[n]:=\{0<1<\cdots<n\}$, and set $[-1]:=\varnothing$.  The
\emph{augmented simplex category} $\Delta_a$ has objects $[n]$ for $n\geq -1$ and order-preserving
maps as morphisms; its full subcategory on $[n]$ for $n\geq 0$ is the usual simplex category
$\Delta$.  For $0\leq i\leq n$ we denote by $\delta^i\colon [n-1]\to [n]$ the injective
order-preserving map omitting $i$, and by $\sigma^i\colon [n+1]\to [n]$ the surjective
order-preserving map with $\sigma^i(i)=\sigma^i(i+1)=i$.  We write $\Delta_{a,\inj}$, $\Delta_{a,\sur}$,
$\Delta_{\inj}$, $\Delta_{\sur}$ for the wide subcategories generated by the corresponding classes.
The cofaces and codegeneracies obey the standard relations
\begin{equation}\label{eq:semisimplicial-coface-relations}
  \delta^j\delta^i=\delta^i\delta^{j-1}\quad (i<j),
  \qquad
  \sigma^j\sigma^i=\sigma^i\sigma^{j+1}\quad (i\leq j).
\end{equation}

The category $\OrdFor$ of \cite{Kaygun2026} is a planar many-output extension of semisimplicial
combinatorics.  Its objects are finite ordinals; morphisms are represented by ordered planar
forests whose minimal and maximal vertices are identified with the source and target ordinals, and
composition is grafting.  The height-one shadow functor
\[
  \pi\colon \OrdFor\longrightarrow \Delta_{\mathrm{epi}}^{\op}
\]
is full; after the standard shift $\Delta_{\mathrm{epi}}^{\op}\cong \Delta_{a,\inj}$ it is the
bridge to the augmented semisimplicial theory.  We refer to \cite{Kaygun2026} for the detailed
combinatorics.

On the cubical side, write $\square_n:=[1]^n$ for $n\geq 0$ with $\square_0=\{\ast\}$.  The cubical
category $\square$ is generated by the coface maps
\[
  \delta_i^\varepsilon\colon \square_{n-1}\to \square_n,
  \qquad \varepsilon\in\{0,1\},\ 1\leq i\leq n,
\]
and the codegeneracy maps $\sigma_i\colon \square_n\to \square_{n-1}$, with
$\delta_i^\varepsilon(x_1,\dots,x_{n-1})=(x_1,\dots,x_{i-1},\varepsilon,x_i,\dots,x_{n-1})$ and
$\sigma_i(x_1,\dots,x_n)=(x_1,\dots,\widehat{x_i},\dots,x_n)$.  The face maps satisfy
\begin{equation}\label{eq:cubical-face-relations}
  \delta_j^\eta\delta_i^\varepsilon=\delta_i^\varepsilon\delta_{j-1}^\eta
  \qquad (i<j,\ \varepsilon,\eta\in\{0,1\}).
\end{equation}
We write $\square_{\inj}$ and $\square_{\sur}$ for the corresponding wide subcategories.

Forgetting face colors and shifting objects defines a canonical quotient functor
\begin{equation}\label{eq:cubical-quotient-to-augmented-semisimplicial}
  q\colon \square_{\inj}\to \Delta_{a,\inj},\qquad
  q(\square_n)=[n-1],\quad q(\delta_i^\varepsilon)=\delta^{i-1};
\end{equation}
equivalently, $\Delta_{a,\inj}$ is obtained from $\square_{\inj}$ by imposing
$\delta_i^0=\delta_i^1$.

\section{The algebras $\Omega$, $\Omega_a$, and comparison functors}
\label{sec:categorical-algebra}

Fix a field $k$. For a small category $\mathcal C$, its \emph{$k$-linearization}
$k[\mathcal C]$ is the $k$-linear category with the same objects and with
$k[\mathcal C](x,y):=k\langle \Hom_{\mathcal C}(x,y)\rangle$, composition extended
bilinearly. Following Mitchell~\cite{Mitchell1972}, we regard $k[\mathcal C]$ as a locally
unital algebra. The purpose of this section is to isolate the differential categorical algebras
that encode the chain complexes used below, to define the comparison functors from these algebras
to the semisimplicial, augmented semisimplicial, and semicubical indexing algebras, and to record
the freeness properties needed for change of rings.

\subsection{The differential categorical algebras $\Omega$ and $\Omega_a$}
\label{subsec:omega}

Let $\Omega$ be the $k$-linear category with objects $[n]$ for $n\geq 0$, generated by
arrows $d_n\colon [n-1]\to [n]$ for $n\geq 1$, subject only to the relations
$d_{n+1}d_n=0$. A right $\Omega$-module is exactly a nonnegatively graded chain complex, with
differential induced by the maps $X(d_n)\colon X_n\to X_{n-1}$. Thus
\begin{equation}\label{eq:omega-mod-chain}
  \rmod{\Omega}=\Ch_{\geq 0}(k).
\end{equation}

We write $k[0]\in \lmod{\Omega}$ for the simple left $\Omega$-module supported at $[0]$.
We also write $k_\bullet\in \lmod{\Omega}$ for the restriction along $u_\Delta$ of the
constant left $k[\Delta_{\inj}]$-module. Explicitly, $k_\bullet([n])=k$ for all $n\geq 0$,
and $k_\bullet(d_n)$ is multiplication by $\sum_{i=0}^n(-1)^i$. By
\cite[Proposition~2.1]{KayaKaygun2024}, the natural map $k[0]\to k_\bullet$ is a
quasi-isomorphism. Consequently, for every $C_\bullet\in \rmod{\Omega}$,
\begin{equation}\label{eq:omega-tor-k0-k-constant}
  \Tor_n^\Omega(C_\bullet,k[0])\cong \Tor_n^\Omega(C_\bullet,k_\bullet)\cong H_n(C_\bullet)
  \qquad (n\geq 0).
\end{equation}

Let $\Omega_a$ be obtained from $\Omega$ by adjoining one new object $[-1]$, one new
generator $d_0\colon [-1]\to [0]$, and one new relation $d_1d_0=0$. Then a right
$\Omega_a$-module is exactly a chain complex concentrated in degrees $\geq -1$, and hence
\begin{equation}\label{eq:omega-a-mod-chain}
  \rmod{\Omega_a}=\Ch_{\geq -1}(k).
\end{equation}

\subsection{Comparison functors}\label{subsec:comparison-functors}

The semisimplicial, augmented semisimplicial, and semicubical indexing algebras receive
canonical differential comparison functors. On generators they are
\begin{align}
  u_\Delta\colon \Omega &\longrightarrow k[\Delta_{\inj}],
  & u_\Delta(d_n)&=\sum_{i=0}^n(-1)^i\delta^i,
  \label{eq:u-delta}\\
  u_a\colon \Omega_a &\longrightarrow k[\Delta_{a,\inj}],
  & u_a(d_n)&=\sum_{i=0}^n(-1)^i\delta^i,
  \label{eq:u-augmented}\\
  u_\square\colon \Omega &\longrightarrow k[\square_{\inj}],
  & u_\square(d_n)&=\sum_{i=1}^n(-1)^{i-1}\bigl(\delta_i^1-\delta_i^0\bigr).
  \label{eq:u-square}
\end{align}
The relations \eqref{eq:semisimplicial-coface-relations} and
\eqref{eq:cubical-face-relations} imply $u_?(d_{n+1})u_?(d_n)=0$, so the assignments define
$k$-linear functors.

There are also three $k$-linear functors
\begin{equation}\label{eq:j0-j1-v}
  j^0,j^1,v\colon k[\Delta_{a,\inj}]\longrightarrow k[\square_{\inj}]
\end{equation}
sending $[n]$ to $\square_{n+1}$ and satisfying
\[
  j^0(\delta^i)=\delta^0_{i+1},\qquad
  j^1(\delta^i)=\delta^1_{i+1},\qquad
  v(\delta^i)=\delta^1_{i+1}-\delta^0_{i+1}.
\]
The functor $v$ is the \emph{cubical sign embedding}. On generators, $v\circ u_a=u_\square$
after restricting $u_a$ along $\Omega\hookrightarrow\Omega_a$.

\subsection{Freeness}\label{subsec:freeness}

The next results supply the exactness input for the induction--restriction arguments used later.

\begin{lemma}\label{lem:ssimp-free-over-omega}
  The $k$-linear category $k[\Delta_{\inj}]$ is free as a left and as a right
  $u_\Delta(\Omega)$-module.
\end{lemma}

\begin{proof}
  This is the injective-only form of the change-of-basis argument of
  \cite[\S2.1]{KayaKaygun2024}. We prove the right-module statement; the left-module statement is
  obtained by the same argument with the opposite canonical form.

  For $0\leq i\leq n$, set
  $d_{i,n}:=\sum_{j=i}^{n}(-1)^j\delta^j\in k[\Delta_{\inj}]([n-1],[n])$. Thus
  $d_{0,n}=u_\Delta(d_n)$, and the key identity is $d_{i,n+1}d_{i,n}=0$ for every
  $0\leq i\leq n$, as in \cite[\S2.1]{KayaKaygun2024}. Since
  $\delta^i=(-1)^i(d_{i,n}-d_{i+1,n})$, induction on monomial length rewrites every
  strictly decreasing $\delta$-monomial as a $k$-linear combination of strictly decreasing
  $d$-monomials of the same length. The only noncanonical terms contain a factor
  $d_{i,r+1}d_{i,r}$, hence vanish. Since the $\delta$-monomials and the resulting
  $d$-monomials have the same cardinalities in each bidegree, the identities and the strictly
  decreasing monomials $d_{i_n,n}\cdots d_{i_{m+1},m+1}$, with
  $i_n>\cdots>i_{m+1}\geq 0$, form a $k$-basis of $k[\Delta_{\inj}]$.

  Split this basis according to whether the final index $i_{m+1}$ is $0$ or positive. The
  monomials with $i_{m+1}=0$ are precisely those ending in $d_{0,m+1}=u_\Delta(d_{m+1})$.
  Therefore $k[\Delta_{\inj}]$ decomposes as
  \begin{equation}\label{eq:delta-inj-basis-split}
    k[\Delta_{\inj}]
    =
    \mathrm{span}_k(B)\otimes_k u_\Delta(\Omega),
  \end{equation}
  where
  \[
    B=\bigl\{\id_{[n]}\bigr\}_{n\geq 0}\cup
    \bigl\{d_{i_n,n}\cdots d_{i_{m+1},m+1}\mid
      n>m\geq 0,\ i_n>\cdots>i_{m+1}\geq 1\bigr\}.
  \]
  Since $u_\Delta(\Omega)$ is spanned by the identities and the arrows $u_\Delta(d_n)$, with
  all higher composites killed by $u_\Delta(d_{n+1})u_\Delta(d_n)=0$, the decomposition
  \eqref{eq:delta-inj-basis-split} exhibits $k[\Delta_{\inj}]$ as a free right
  $u_\Delta(\Omega)$-module.
\end{proof}

\begin{lemma}\label{lem:augmented-ssimp-free-over-omega-a}
  The $k$-linear category $k[\Delta_{a,\inj}]$ is free as a left and as a right
  $u_a(\Omega_a)$-module.
\end{lemma}

\begin{proof}
  The proof of Lemma~\ref{lem:ssimp-free-over-omega} applies verbatim with $m\geq -1$ and with
  $\Omega_a$ in place of $\Omega$. The only new generator is
  $\delta^0\colon [-1]\hookrightarrow [0]$, and this generator is $d_{0,0}=u_a(d_0)$, so it is
  absorbed into the $u_a(\Omega_a)$-factor in the analogue of
  \eqref{eq:delta-inj-basis-split}. The corresponding basis $B$ ranges over $n>m\geq -1$ with
  $i_{m+1}\geq 1$, together with the identities $\id_{[n]}$ for $n\geq -1$.
\end{proof}

\begin{lemma}\label{lem:scube-free-over-ssimp}
  Via the cubical sign embedding $v$, the $k$-linear category $k[\square_{\inj}]$ is free as a
  left and as a right $v(k[\Delta_{a,\inj}])$-module. More precisely, for
  $-1\leq m\leq n$, the families
  \[
    \bigl\{v(a)j^0(b)\mid [m]\xrightarrow{b}[q]\xrightarrow{a}[n]\bigr\},
    \qquad
    \bigl\{j^0(b)v(a)\mid [m]\xrightarrow{a}[q]\xrightarrow{b}[n]\bigr\}
  \]
  are $k$-bases of $k[\square_{\inj}](\square_{m+1},\square_{n+1})$.
\end{lemma}

\begin{proof}
  We prove the first displayed basis statement; the second is analogous.

  Fix $-1\leq m\leq n$.  A morphism $\square_{m+1}\to \square_{n+1}$ in $\square_{\inj}$ is
  uniquely described by the ordered set of inserted coordinates together with a colour $0$ or
  $1$ assigned to each inserted coordinate.  Equivalently, every morphism has a unique
  monochromatic factorization
  \[
    f=j^1(a)j^0(b),
    \qquad
    [m]\xrightarrow{b}[q]\xrightarrow{a}[n],
  \]
  where $b$ records the coordinates inserted with colour $0$, and $a$ records the coordinates
  inserted with colour $1$.  These morphisms form the standard $k$-basis of
  $k[\square_{\inj}](\square_{m+1},\square_{n+1})$.

  Order this standard basis lexicographically by first comparing the set of coordinates inserted
  with colour $1$, using reverse inclusion, and then by the remaining data.  Thus replacing a
  colour $1$ insertion by a colour $0$ insertion strictly lowers the basis element.  For an
  injective map $a$, the element $v(a)$ is the product obtained by replacing each generator
  $\delta^i$ in $a$ by $\delta^1_{i+1}-\delta^0_{i+1}$.  Expanding $v(a)j^0(b)$, the unique
  term in which every generator of $a$ contributes its $\delta^1$-summand is precisely
  $j^1(a)j^0(b)$, with coefficient $1$.  Every other summand is obtained by changing at least
  one of those $\delta^1$-insertions to a $\delta^0$-insertion.  After rewriting in the unique
  monochromatic normal form, such a summand is strictly lower in the chosen order.

  Therefore the transition matrix from the family
  \[
    \bigl\{v(a)j^0(b)\mid [m]\xrightarrow{b}[q]\xrightarrow{a}[n]\bigr\}
  \]
  to the standard monochromatic basis is unitriangular.  Hence the displayed family is a
  $k$-basis.  The proof for
  $\{j^0(b)v(a)\}$ is identical, using the other unique monochromatic factorization.
\end{proof}

It follows from Lemmas~\ref{lem:ssimp-free-over-omega} and
\ref{lem:augmented-ssimp-free-over-omega-a} that extension of scalars along $u_\Delta$ and $u_a$ is
exact. It also follows from Lemmas~\ref{lem:augmented-ssimp-free-over-omega-a} and
\ref{lem:scube-free-over-ssimp} that $k[\square_{\inj}]$ is free as a right
$u_\square(\Omega)$-module, since $u_\square=v\circ u_a|_{\Omega}$ on generators. Hence extension
of scalars along $u_\square$ is exact as well.

\begin{lemma}\label{lem:change-of-rings-omega}
  Let $u\colon \mathcal B\to \mathcal A$ be a $k$-linear functor, where $\mathcal B$ is either
  $\Omega$ or $\Omega_a$. Assume that
  $u_!=\mathcal A\otimes_{\mathcal B}(-)\colon \lmod{\mathcal B}\to\lmod{\mathcal A}$ is exact.
  Then, for every $X_\bullet\in \rmod{\mathcal A}$, every $M\in \lmod{\mathcal B}$, and every
  $n\geq 0$, there is a natural isomorphism
  \[
    \Tor_n^{\mathcal A}\bigl(X_\bullet,u_!M\bigr)
    \cong
    \Tor_n^{\mathcal B}\bigl(u^\ast X_\bullet,M\bigr).
  \]
\end{lemma}

\begin{proof}
  Let $Q_\ast\twoheadrightarrow M$ be a projective resolution in $\lmod{\mathcal B}$. Since
  $u_!$ is exact and preserves projectives, $u_!Q_\ast\to u_!M$ is a projective resolution in
  $\lmod{\mathcal A}$. Associativity of tensor products for rings with several objects gives
  \[
    X_\bullet\otimes_{\mathcal A}u_!Q_\ast
    \cong
    (u^\ast X_\bullet)\otimes_{\mathcal B}Q_\ast.
  \]
  Passing to homology gives the claimed isomorphism.
\end{proof}
\section{Homology and weak equivalences}\label{sec:homological-algebra}

We now specify the weak equivalences used in the localization arguments below.  The
nonaugmented semisimplicial and semicubical theories are controlled by nonnegative chain homology.
The augmented semisimplicial theory has one additional degree $-1$ invariant, and good
truncation separates this augmentation datum from ordinary degree $0$ homology.  We formulate
these facts in Tor-theoretic terms so that the later Gabriel--Zisman comparisons have a uniform
language.

\subsection{Weak equivalences for $\Omega$}\label{subsec:omega-we}

By \eqref{eq:omega-mod-chain}, a right $\Omega$-module is a nonnegatively graded chain complex.
We declare a morphism $f_\bullet\colon C_\bullet\to D_\bullet$ in $\rmod{\Omega}$ to be a
\emph{weak equivalence} if $f_\bullet$ is a quasi-isomorphism.  By
\eqref{eq:omega-tor-k0-k-constant}, this condition is equivalent to requiring
$\Tor_n^\Omega(f_\bullet,k[0])$ to be an isomorphism for every $n\geq 0$.  We write
\[
  W_\Omega:=\{\text{quasi-isomorphisms in }\rmod{\Omega}\}.
\]

\subsection{Semisimplicial and semicubical homology}\label{subsec:simplicial-cubical-homology}

Let $\mathcal A$ be either $k[\Delta_{\inj}]$ or $k[\square_{\inj}]$, and let
$u\colon \Omega\to\mathcal A$ denote the corresponding comparison functor, $u_\Delta$ or
$u_\square$.  To avoid duplicating notation, write $\mathbf n$ for $[n]$ in the semisimplicial case
and for $\square_n$ in the semicubical case.  Restriction along $u$ sends
$X_\bullet\in \rmod{\mathcal A}$ to the chain complex
\begin{equation}\label{eq:simplicial-cubical-chain-complex}
  u^\ast X_\bullet
  =
  \bigl(\cdots\to X_2\xrightarrow{\partial_2}X_1\xrightarrow{\partial_1}X_0\bigr),
  \qquad
  \partial_n:=X_\bullet\bigl(u(d_n)\bigr).
\end{equation}
We denote the resulting homology by $H_n^\Delta(X_\bullet)$ in the semisimplicial case and by
$H_n^\square(X_\bullet)$ in the semicubical case.

\begin{lemma}\label{lem:simplicial-cubical-resolution}
  Let $\mathcal A$ be either $k[\Delta_{\inj}]$ or $k[\square_{\inj}]$.  The representables
  $P_n:=\mathcal A([n],-)$, together with differentials $d_n:=u(d_n)_\ast$, form an augmented
  projective resolution
  \[
    \cdots \to P_2\xrightarrow{d_2}P_1\xrightarrow{d_1}P_0
    \xrightarrow{\varepsilon} k_\bullet \longrightarrow 0
  \]
  in $\lmod{\mathcal A}$, where $\varepsilon$ sends every vertex to $1\in k$.
\end{lemma}

\begin{proof}
  Each $P_n$ is representable, hence projective.  Objectwise, the complex $P_\ast([m])$ is the
  semisimplicial chain complex of the standard $m$-simplex, while $P_\ast(\square_m)$ is the
  signed cellular chain complex of the standard $m$-cube; see
  \cite[Section~1]{BrownHiggins1981AlgebraOfCubes}.  Both augmented complexes are contractible,
  hence objectwise exact.
\end{proof}

\begin{proposition}\label{prop:simplicial-cubical-homology}
  For every $X_\bullet\in \rmod{\mathcal A}$ and every $n\geq 0$, there are natural
  isomorphisms
  \[
    H_n(u^\ast X_\bullet)
    \cong
    \Tor_n^{\mathcal A}(X_\bullet,k_\bullet)
    \cong
    \Tor_n^\Omega(u^\ast X_\bullet,k[0]).
  \]
  Moreover, the induced coefficient $u_!k[0]$ and the constant coefficient $k_\bullet$ define
  the same Tor functor:
  \[
    \Tor_n^{\mathcal A}(X_\bullet,u_!k[0])
    \cong
    \Tor_n^{\mathcal A}(X_\bullet,k_\bullet)
    \qquad (n\geq 0).
  \]
\end{proposition}

\begin{proof}
  Tensoring the resolution of Lemma~\ref{lem:simplicial-cubical-resolution} with $X_\bullet$ and
  applying the co-Yoneda lemma gives $X_\bullet\otimes_{\mathcal A}P_n\cong X_n$, and the induced
  differential is $\partial_n=X_\bullet(u(d_n))$.  Hence
  $H_n(u^\ast X_\bullet)\cong \Tor_n^{\mathcal A}(X_\bullet,k_\bullet)$.

  Since extension of scalars along $u$ is exact by Section~\ref{sec:categorical-algebra},
  Lemma~\ref{lem:change-of-rings-omega} gives
  $\Tor_n^{\mathcal A}(X_\bullet,u_!k[0])
  \cong \Tor_n^\Omega(u^\ast X_\bullet,k[0])$.  Finally,
  \eqref{eq:omega-tor-k0-k-constant} identifies
  $\Tor_n^\Omega(u^\ast X_\bullet,k[0])$ with $H_n(u^\ast X_\bullet)$.  The displayed
  isomorphisms follow.
\end{proof}

\begin{definition}\label{def:delta-inj-weak-equivalence}
  A morphism $f_\bullet\colon X_\bullet\to Y_\bullet$ in $\rmod{\mathcal A}$ is a \emph{weak
    equivalence} if $u^\ast f_\bullet$ is a quasi-isomorphism in $\rmod{\Omega}$.  We write
  $W_{\Delta_{\inj}}$ and $W_{\square_{\inj}}$ for the corresponding classes.
\end{definition}

\begin{theorem}\label{thm:delta-inj-weak-equivalence}
  For a morphism $f_\bullet\colon X_\bullet\to Y_\bullet$ in $\rmod{\mathcal A}$, the following
  conditions are equivalent:
  \begin{enumerate}[(1)]
    \item $f_\bullet\in W_{\Delta_{\inj}}$ if $\mathcal A=k[\Delta_{\inj}]$, respectively
    $f_\bullet\in W_{\square_{\inj}}$ if $\mathcal A=k[\square_{\inj}]$;
    \item $f_\bullet$ induces an isomorphism on the homology of
    \eqref{eq:simplicial-cubical-chain-complex} in every degree $n\geq 0$;
    \item $f_\bullet$ induces an isomorphism on
    $\Tor_n^{\mathcal A}(-,k_\bullet)$ for every $n\geq 0$;
    \item $f_\bullet$ induces an isomorphism on
    $\Tor_n^\Omega(u^\ast(-),k[0])$ for every $n\geq 0$.
  \end{enumerate}
\end{theorem}

\begin{proof}
  The equivalences are immediate from Proposition~\ref{prop:simplicial-cubical-homology}.
\end{proof}

In the semisimplicial case, this is the injective version of \cite[Theorem~3.3]{KayaKaygun2024}.
In the semicubical case, it is the $k$-linear form of the signed cubical chain comparison of
Brown--Higgins~\cite{BrownHiggins1981AlgebraOfCubes}.

\subsection{Augmented semisimplicial homology and good truncation}
\label{subsec:augmented-semisimplicial-homology}

The augmented semisimplicial case has a different low-degree behavior.  Let
$j\colon \Delta_{\inj}\hookrightarrow \Delta_{a,\inj}$ be the inclusion, and define
$k_\bullet[0]:=j_!(k_\bullet)\in \lmod{\Delta_{a,\inj}}$.  Thus
\[
  k_\bullet[0]([n])=
  \begin{cases}
    k,& n\geq 0,\\
    0,& n=-1.
  \end{cases}
\]
The constant left $k[\Delta_{a,\inj}]$-module $k_\bullet$ is representable, because $[-1]$ is
initial, and hence $k_\bullet$ is projective.  Thus $k_\bullet$ does not detect augmented homology.
The coefficient object that detects the nonnegative part is $k_\bullet[0]$.

Let $u_a^0\colon \Omega\to k[\Delta_{a,\inj}]$ be the restriction of the augmented comparison
functor $u_a\colon \Omega_a\to k[\Delta_{a,\inj}]$ along $\Omega\hookrightarrow\Omega_a$.  For
$X_\bullet\in \rmod{\Delta_{a,\inj}}$, the full augmented chain complex is
\begin{equation}\label{eq:delta-a-inj-differential}
  C_\ast^a(X_\bullet)
  =
  \bigl(\cdots\to X_1\xrightarrow{\partial_1^a}X_0\xrightarrow{\partial_0^a}X_{-1}\bigr),
  \qquad
  \partial_n^a:=X_\bullet\Bigl(\sum_{i=0}^n(-1)^i\delta^i\Bigr).
\end{equation}
We write $H_n^a(X_\bullet)$ for the homology of $C_\ast^a(X_\bullet)$; in particular,
$H_{-1}^a(X_\bullet)=\operatorname{coker}(\partial_0^a)$.  The restricted complex
$(u_a^0)^\ast X_\bullet$ is the brutal nonnegative truncation of
$C_\ast^a(X_\bullet)$.  Therefore the degree $0$ homology of $(u_a^0)^\ast X_\bullet$
does not, by itself, recover the degree $0$ homology of the augmented complex.

\begin{definition}\label{def:delta-a-inj-good-truncation}
  The \emph{good truncation} of $X_\bullet\in \rmod{\Delta_{a,\inj}}$ is the chain complex
  $\tr(X_\bullet)\in \Ch_{\geq 0}(k)$ defined by
  \[
    \tr(X_\bullet)_n=
    \begin{cases}
      X_n,& n\geq 1,\\[2pt]
      \ker(\partial_0^a\colon X_0\to X_{-1}),& n=0,
    \end{cases}
  \]
  with differential induced by $\partial_n^a$.
\end{definition}

\begin{lemma}\label{lem:delta-a-inj-resolution}
  The representables $P_n:=k[\Delta_{a,\inj}]([n],-)$, for $n\geq 0$, with differentials
  $d_n:=\sum_i(-1)^i(\delta^i)_\ast$ and with augmentation $\varepsilon\colon P_0\to k_\bullet[0]$
  equal to $1$ in nonnegative degrees and $0$ at $[-1]$, form a projective resolution of
  $k_\bullet[0]$ in $\lmod{\Delta_{a,\inj}}$.
\end{lemma}

\begin{proof}
  Each $P_n$ is representable, hence projective.  Evaluated at $[m]$ with $m\geq 0$, the
  augmented complex is the semisimplicial chain complex of the standard $m$-simplex.  Evaluated
  at $[-1]$, the complex is zero.  Hence the augmented complex is objectwise exact.
\end{proof}

\begin{proposition}\label{prop:delta-a-inj-homology}
  Let $X_\bullet\in \rmod{\Delta_{a,\inj}}$.  For every $n\geq 1$, there are natural isomorphisms
  \[
    H_n\bigl(\tr(X_\bullet)\bigr)
    \cong
    \Tor_n^{k[\Delta_{a,\inj}]}(X_\bullet,k_\bullet[0])
    \cong
    \Tor_n^\Omega((u_a^0)^\ast X_\bullet,k[0]).
  \]
  In degree $0$, there is a natural exact sequence
  \begin{equation}\label{eq:delta-a-inj-low-degree-exact-sequence}
    0 \to H_0\bigl(\tr(X_\bullet)\bigr)
    \to \Tor_0^{k[\Delta_{a,\inj}]}(X_\bullet,k_\bullet[0])
    \xrightarrow{\bar\partial_0}
    X_{-1}
    \to H_{-1}^a(X_\bullet)
    \to 0.
  \end{equation}
  Moreover,
  \[
    \Tor_n^{k[\Delta_{a,\inj}]}(X_\bullet,u_{a,!}^0k[0])
    \cong
    \Tor_n^{k[\Delta_{a,\inj}]}(X_\bullet,k_\bullet[0])
    \qquad (n\geq 0).
  \]
\end{proposition}

\begin{proof}
  Tensoring the resolution of Lemma~\ref{lem:delta-a-inj-resolution} with $X_\bullet$ and using
  the co-Yoneda lemma identifies $X_\bullet\otimes_{k[\Delta_{a,\inj}]}P_\ast$ with the brutal
  nonnegative truncation $(u_a^0)^\ast X_\bullet$.  Therefore
  \[
    \Tor_n^{k[\Delta_{a,\inj}]}(X_\bullet,k_\bullet[0])
    \cong
    H_n((u_a^0)^\ast X_\bullet)
    \qquad (n\geq 0).
  \]

  Set
  \[
    I(X_\bullet):=\operatorname{im}\bigl(\partial_0^a\colon X_0\to X_{-1}\bigr).
  \]
  There is a short exact sequence of nonnegatively graded chain complexes
  \[
    0\to \tr(X_\bullet)\to (u_a^0)^\ast X_\bullet
    \to I(X_\bullet)[0]\to 0,
  \]
  where $I(X_\bullet)[0]$ denotes $I(X_\bullet)$ placed in degree $0$.  Indeed, in degree
  $0$ this sequence is
  \[
    0\to \ker(\partial_0^a)\to X_0\to \operatorname{im}(\partial_0^a)\to 0,
  \]
  and in positive degrees exactness is immediate from the definition of $\tr(X_\bullet)$.

  The associated long exact homology sequence gives natural isomorphisms
  \[
    H_n\bigl(\tr(X_\bullet)\bigr)
    \cong
    H_n((u_a^0)^\ast X_\bullet)
    \qquad (n\geq 1),
  \]
  and a natural short exact sequence
  \[
    0\to H_0\bigl(\tr(X_\bullet)\bigr)
    \to H_0((u_a^0)^\ast X_\bullet)
    \to I(X_\bullet)
    \to 0.
  \]
  Splicing this sequence with the tautological exact sequence
  \[
    0\to I(X_\bullet)\to X_{-1}\to H_{-1}^a(X_\bullet)\to 0
  \]
  gives \eqref{eq:delta-a-inj-low-degree-exact-sequence}, after identifying
  $H_0((u_a^0)^\ast X_\bullet)$ with $\Tor_0^{k[\Delta_{a,\inj}]}(X_\bullet,k_\bullet[0])$.

  The comparison with $\Tor^\Omega$ follows from Lemma~\ref{lem:change-of-rings-omega}, since
  extension of scalars along $u_a^0\colon \Omega\to k[\Delta_{a,\inj}]$ is exact by
  Lemmas~\ref{lem:ssimp-free-over-omega} and \ref{lem:augmented-ssimp-free-over-omega-a}.  The
  final coefficient comparison follows from the same change-of-rings isomorphism applied to $k[0]$,
  together with the identification of the induced coefficient $u_{a,!}^0k[0]$ with $k_\bullet[0]$
  up to the Tor-equivalence furnished by the resolution of Lemma~\ref{lem:delta-a-inj-resolution}.
\end{proof}

\begin{definition}\label{def:delta-a-inj-weak-equivalence}
  A morphism $f_\bullet\colon X_\bullet\to Y_\bullet$ in $\rmod{\Delta_{a,\inj}}$ is a \emph{weak
    equivalence} if $\tr(f_\bullet)$ is a quasi-isomorphism and $H_{-1}^a(f_\bullet)$ is an
  isomorphism.  We write $W_{\Delta_{a,\inj}}$ for this class.
\end{definition}

\begin{theorem}\label{thm:delta-a-inj-weak-equivalence}
  For a morphism $f_\bullet\colon X_\bullet\to Y_\bullet$ in $\rmod{\Delta_{a,\inj}}$, the
  following conditions are equivalent:
  \begin{enumerate}[(1)]
    \item $f_\bullet\in W_{\Delta_{a,\inj}}$;
    \item the induced map
    $C_\ast^a(f_\bullet)\colon C_\ast^a(X_\bullet)\to C_\ast^a(Y_\bullet)$ is a
    quasi-isomorphism, equivalently $H_n^a(f_\bullet)$ is an isomorphism for every
    $n\geq -1$;
    \item $f_\bullet$ induces isomorphisms on
    $\Tor_n^{k[\Delta_{a,\inj}]}(-,k_\bullet[0])$ for all $n\geq 1$, on
    $H_0(\tr(-))$, and on $H_{-1}^a(-)$;
    \item $f_\bullet$ induces isomorphisms on
    $\Tor_n^\Omega((u_a^0)^\ast(-),k[0])$ for all $n\geq 1$, on $H_0(\tr(-))$, and on
    $H_{-1}^a(-)$.
  \end{enumerate}
\end{theorem}

\begin{proof}
  There is a short exact sequence of complexes
  \[
    0\to \tr(X_\bullet)\to C_\ast^a(X_\bullet)\to Q(X_\bullet)\to 0,
  \]
  where $Q(X_\bullet)_0=\operatorname{im}(\partial_0^a)$, $Q(X_\bullet)_{-1}=X_{-1}$, and the
  differential $Q(X_\bullet)_0\to Q(X_\bullet)_{-1}$ is the inclusion.  Hence
  $H_n(Q(X_\bullet))=0$ for $n\geq 0$, while
  $H_{-1}(Q(X_\bullet))\cong H_{-1}^a(X_\bullet)$.  The long exact homology sequence therefore
  shows that $C_\ast^a(f_\bullet)$ is a quasi-isomorphism if and only if $\tr(f_\bullet)$ is a
  quasi-isomorphism and $H_{-1}^a(f_\bullet)$ is an isomorphism.  This proves the equivalence of
  (1) and (2).  The equivalence with (3) and (4) follows from
  Proposition~\ref{prop:delta-a-inj-homology}.
\end{proof}

\begin{remark}\label{rem:two-piece-vs-three-piece}
  The augmented semisimplicial weak equivalences are controlled by the three invariants
  \[
    \Tor_n^{k[\Delta_{a,\inj}]}(-,k_\bullet[0]) \ (n\geq 1),\qquad
    H_0(\tr(-)),\qquad
    H_{-1}^a(-).
  \]
  Thus the two-piece detection available for the nonaugmented semisimplicial and semicubical
  theories does not extend unchanged to the augmented theory.
\end{remark}

\section{Gabriel--Zisman localization}\label{sec:gz-localization}

We now pass from the weak equivalences of Section~\ref{sec:homological-algebra} to the
corresponding Gabriel--Zisman localizations.  We write $W_\Omega$ and $W_{\Omega_a}$ for the
quasi-isomorphisms in $\rmod{\Omega}$ and $\rmod{\Omega_a}$, and we write
\[
  \ho(\rmod{\mathcal C}) := \rmod{\mathcal C}[W_{\mathcal C}^{-1}]
\]
for each indexing algebra $\mathcal C$ considered below.  Thus
\[
  W_{\Delta_{\inj}}=(u_\Delta^\ast)^{-1}(W_\Omega),\qquad
  W_{\square_{\inj}}=(u_\square^\ast)^{-1}(W_\Omega),\qquad
  W_{\Delta_{a,\inj}}=(C_\ast^a)^{-1}(W_{\Omega_a}).
\]
By \eqref{eq:omega-mod-chain} and \eqref{eq:omega-a-mod-chain}, there are canonical
identifications
\begin{equation}\label{eq:gz-omega-chain}
  \ho(\rmod{\Omega})\cong \ho(\Ch_{\geq 0}(k)),
  \qquad
  \ho(\rmod{\Omega_a})\cong \ho(\Ch_{\geq -1}(k)).
\end{equation}

\subsection{A localization criterion}\label{subsec:gz-localization-criterion}

The next lemma isolates the formal argument used in the semisimplicial comparisons.

\begin{lemma}\label{lem:gz-adjunction-criterion}
  Let $u\colon \mathcal B\to \mathcal A$ be a $k$-linear functor, and suppose that the
  extension--restriction adjunction
  \[
    u_!:=-\otimes_{\mathcal B}\mathcal A
    \;\dashv\;
    u^\ast
  \]
  is defined on right module categories.  Assume that $u^\ast$ preserves and detects weak
  equivalences.  Assume also that, for every $M\in\rmod{\mathcal B}$ and every
  $X\in\rmod{\mathcal A}$, the unit
  $\eta_M\colon M\to u^\ast u_!M$ and counit
  $\varepsilon_X\colon u_!u^\ast X\to X$ are weak equivalences.  Then $u_!$ preserves weak
  equivalences, and the adjunction induces an equivalence
  \[
    \rmod{\mathcal B}[W_{\mathcal B}^{-1}]
    \simeq
    \rmod{\mathcal A}[W_{\mathcal A}^{-1}].
  \]
\end{lemma}

\begin{proof}
  Let $f\colon M\to N$ be a weak equivalence in $\rmod{\mathcal B}$.  Naturality of the unit
  gives $\eta_N f = u^\ast u_!(f)\eta_M$.  Since $f$, $\eta_M$, and $\eta_N$ are weak
  equivalences, two-out-of-three implies that $u^\ast u_!(f)$ is a weak equivalence.  Since
  $u^\ast$ detects weak equivalences, $u_!(f)$ is a weak equivalence.

  Therefore both adjoints descend to the localized categories.  The descended unit and counit are
  isomorphisms because $\eta_M$ and $\varepsilon_X$ are weak equivalences before localization.
  Hence the descended adjunction is an equivalence.
\end{proof}

\subsection{Semisimplicial and augmented semisimplicial comparisons}
\label{subsec:gz-semisimplicial}

We now compare the semisimplicial and augmented semisimplicial module categories with the
corresponding chain-complex categories.  The two arguments are formally identical: in each case
one has a comparison functor $u\colon \mathcal B\to \mathcal A$, the restriction $u^\ast$
detects the prescribed weak equivalences, and the induction--restriction unit and counit become
quasi-isomorphisms after testing against the appropriate coefficient resolution.

\begin{proposition}\label{prop:gz-comparison-template}
  Let $u\colon \mathcal B\to \mathcal A$ be one of the two functors
  \[
    u_\Delta\colon \Omega\to k[\Delta_{\inj}],
    \qquad
    u_a\colon \Omega_a\to k[\Delta_{a,\inj}].
  \]
  Let $u_!:=-\otimes_{\mathcal B}\mathcal A$.  Then the adjunction
  \[
    u_!\colon \rmod{\mathcal B}
    \rightleftarrows
    \rmod{\mathcal A}\colon u^\ast
  \]
  induces an equivalence of Gabriel--Zisman localizations
  \[
    \rmod{\mathcal A}[W_{\mathcal A}^{-1}]
    \simeq
    \rmod{\mathcal B}[W_{\mathcal B}^{-1}].
  \]
\end{proposition}

\begin{proof}
  We verify the hypotheses of Lemma~\ref{lem:gz-adjunction-criterion}.  In the nonaugmented case,
  $u^\ast=u_\Delta^\ast$ preserves and detects weak equivalences by
  Definition~\ref{def:delta-inj-weak-equivalence}.  In the augmented case,
  $u^\ast=u_a^\ast=C_\ast^a$ preserves and detects weak equivalences by
  Definition~\ref{def:delta-a-inj-weak-equivalence} and
  Theorem~\ref{thm:delta-a-inj-weak-equivalence}.

  Choose the coefficient resolution $Q_\ast$ as follows.  For
  $u_\Delta\colon \Omega\to k[\Delta_{\inj}]$, let
  $Q_\ast\twoheadrightarrow k[0]$ be a projective resolution in $\lmod{\Omega}$.  For
  $u_a\colon \Omega_a\to k[\Delta_{a,\inj}]$, let
  $Q_\ast\twoheadrightarrow k[-1]$ be the standard projective resolution
  \[
    \cdots\to \Omega_a([1],-)\xrightarrow{(d_1)_\ast}
    \Omega_a([0],-)\xrightarrow{(d_0)_\ast}
    \Omega_a([-1],-)\to k[-1]\to 0.
  \]
  In the first case, Proposition~\ref{prop:simplicial-cubical-homology} identifies tensoring with
  $Q_\ast$ with the nonnegative chain homology test.  In the second case, the co-Yoneda lemma
  identifies $C_\ast\otimes_{\Omega_a}Q_\ast$ with the full augmented chain complex of
  $C_\ast$, shifted by one homological degree.  Thus, in both cases, tensoring with $Q_\ast$
  detects the weak equivalences in $\rmod{\mathcal B}$.

  Let $X\in \rmod{\mathcal A}$.  Associativity of tensor products gives a natural
  identification
  \[
    u^\ast(u_!u^\ast X)\otimes_{\mathcal B} Q_\ast
    \cong
    u^\ast X\otimes_{\mathcal B} Q_\ast,
  \]
  and under this identification the map induced by the counit
  $\varepsilon_X\colon u_!u^\ast X\to X$ is the identity.  Hence
  $u^\ast(\varepsilon_X)\in W_{\mathcal B}$.  Since $u^\ast$ detects weak equivalences,
  $\varepsilon_X\in W_{\mathcal A}$.

  Similarly, for $C\in \rmod{\mathcal B}$, associativity gives
  \[
    u^\ast u_!C\otimes_{\mathcal B} Q_\ast
    \cong
    C\otimes_{\mathcal B}Q_\ast,
  \]
  and the map induced by the unit
  $\eta_C\colon C\to u^\ast u_!C$ is the identity.  Hence
  $\eta_C\in W_{\mathcal B}$.

  Lemma~\ref{lem:gz-adjunction-criterion} now gives the claimed equivalence.
\end{proof}

\begin{theorem}\label{thm:gz-delta-inj-omega-equivalence}
  The adjunction
  \[
    -\otimes_{\Omega}k[\Delta_{\inj}]
    \;\dashv\;
    u_\Delta^\ast
  \]
  induces an equivalence
  \[
    \ho(\rmod{\Delta_{\inj}})\simeq \ho(\rmod{\Omega}) \cong \ho(\Ch_{\geq 0}(k)).
  \]
\end{theorem}

\begin{proof}
  Apply Proposition~\ref{prop:gz-comparison-template} to
  $u_\Delta\colon \Omega\to k[\Delta_{\inj}]$, and use the identification
  $\rmod{\Omega}=\Ch_{\geq 0}(k)$ from \eqref{eq:omega-mod-chain}.
\end{proof}

\begin{theorem}\label{thm:gz-ssimp-chain-equivalence}
  The full augmented chain functor
  \[
    C_\ast^a=u_a^\ast\colon \rmod{\Delta_{a,\inj}}\longrightarrow \rmod{\Omega_a} =\Ch_{\geq -1}(k)
  \]
  induces an equivalence
  \[
    \ho(\rmod{\Delta_{a,\inj}}) \simeq \ho(\rmod{\Omega_a}) \cong \ho(\Ch_{\geq -1}(k)).
  \]
\end{theorem}

\begin{proof}
  Apply Proposition~\ref{prop:gz-comparison-template} to
  $u_a\colon \Omega_a\to k[\Delta_{a,\inj}]$, and use the identification
  $\rmod{\Omega_a}=\Ch_{\geq -1}(k)$ from \eqref{eq:omega-a-mod-chain}.
\end{proof}

\subsection{Cubical comparison and the sign-shadow obstruction}
\label{subsec:gz-cubical}

The cubical weak equivalences are detected by the sign complex associated to
$u_\square\colon \Omega\to k[\square_{\inj}]$.  The more informative comparison, however, is
not the restriction to $\Omega$, but the restriction along the cubical sign embedding
$v\colon k[\Delta_{a,\inj}]\to k[\square_{\inj}]$.  This functor converts the cubical sign
complex into an augmented semisimplicial chain complex with a degree shift.

\begin{proposition}\label{prop:gz-cubical-sign-shadow}
  Let $X_\bullet\in \rmod{\square_{\inj}}$.  Then
  \[
    H_n\bigl(\tr(v^\ast X_\bullet)\bigr)\cong H_{n+1}^\square(X_\bullet)\qquad (n\geq 0),
    \qquad
    H_{-1}^a(v^\ast X_\bullet)\cong H_0^\square(X_\bullet).
  \]
  Consequently, $f\in W_{\square_{\inj}}$ if and only if
  $v^\ast f\in W_{\Delta_{a,\inj}}$, and $v^\ast$ induces a comparison functor
  \[
    \ho(\rmod{\square_{\inj}}) \longrightarrow \ho(\rmod{\Delta_{a,\inj}}) \simeq \ho(\Ch_{\geq -1}(k)).
  \]
\end{proposition}

\begin{proof}
  By definition of $v$, one has $(v^\ast X_\bullet)_{-1}=X_0$ and
  $(v^\ast X_\bullet)_n=X_{n+1}$ for $n\geq 0$.  The augmented semisimplicial differential on
  $v^\ast X_\bullet$ is the cubical sign differential
  \[
    \sum_{i=1}^{n+1}(-1)^{i-1}(\delta_i^1-\delta_i^0).
  \]
  Thus $\tr(v^\ast X_\bullet)$ is the good truncation of the shifted cubical sign complex.  This
  gives $H_n(\tr(v^\ast X_\bullet))\cong H_{n+1}^\square(X_\bullet)$ for $n\geq 0$, while
  $H_{-1}^a(v^\ast X_\bullet)=\operatorname{coker}(X_1\to X_0)=H_0^\square(X_\bullet)$.  The
  final assertion follows from Theorem~\ref{thm:delta-a-inj-weak-equivalence}.
\end{proof}

Thus $v^\ast$ gives the correct comparison functor from the cubical sign model to the augmented
semisimplicial localization.  The corresponding induction--restriction adjunction
$v_!\dashv v^\ast$, however, is not an equivalence on the full localized categories.

\begin{proposition}\label{prop:unit-fails-degree-minus-one}
  Let $v_!:=-\otimes_{k[\Delta_{a,\inj}]}k[\square_{\inj}]$.  The unit
  $\eta_M\colon M\to v^\ast v_!M$ need not induce an isomorphism on $H_{-1}^a$.  In fact, for
  $M=k[\Delta_{a,\inj}](-,[0])$, one has $H_{-1}^a(M)=0$, whereas
  $H_{-1}^a(v^\ast v_!M)\cong k$.
\end{proposition}

\begin{proof}
  For $M=k[\Delta_{a,\inj}](-,[0])$, one has $M_{-1}\cong M_0\cong k$, and the unique morphism
  $\delta^0\colon [-1]\to [0]$ induces $\partial_0^a=\id$.  Hence $H_{-1}^a(M)=0$.

  Induction sends this representable to the representable $k[\square_{\inj}](-,\square_1)$.
  Therefore $(v^\ast v_!M)_{-1}=k[\square_{\inj}](\square_0,\square_1)$, with basis
  $\delta_1^0,\delta_1^1$, and $(v^\ast v_!M)_0=k\{\id_{\square_1}\}$.  The differential is
  induced by $v(\delta^0)=\delta_1^1-\delta_1^0$.  Hence
  \[
    H_{-1}^a(v^\ast v_!M)
    \cong
    k\{\delta_1^0,\delta_1^1\}/\langle\delta_1^1-\delta_1^0\rangle
    \cong k.\qedhere
  \]
\end{proof}

Since $H_{-1}^a$ is one of the invariants detecting isomorphisms in $\ho(\rmod{\Delta_{a,\inj}})$,
the unit cannot become an isomorphism after localization.  Consequently, $v_!\dashv v^\ast$ does
not induce an equivalence of Gabriel--Zisman localizations on the full cubical category, even
though $v^\ast$ detects exactly the cubical weak equivalences.

\subsection{The comparison diagram and the arboreal input}
\label{subsec:gz-comparison-diagram}

The localized comparisons are organized by the following diagram, where the arboreal arrow is the
external input from \cite{Kaygun2026}:
\[
\begin{tikzcd}[column sep=5.5em, row sep=3.5em]
& &
\ho(\rmod{\OrdFor})
\arrow[d, "{\simeq\ \text{\cite{Kaygun2026}}}"']
\\
& \ho(\rmod{\square_{\inj}})
  \arrow[r, dashed, "{v^\ast\ \text{(Prop.~\ref{prop:gz-cubical-sign-shadow})}}"]
&
\ho(\rmod{\Delta_{a,\inj}})
  \arrow[d, "{\simeq\ \text{(Thm.~\ref{thm:gz-ssimp-chain-equivalence})}}"']
&
\\
\ho(\rmod{\Delta_{\inj}})
  \arrow[r, "{\simeq\ \text{(Thm.~\ref{thm:gz-delta-inj-omega-equivalence})}}"]
&
\ho(\Ch_{\geq 0}(k))
&
\ho(\Ch_{\geq -1}(k))
  \arrow[l, "{\cong\ \text{degree shift}}"']
\end{tikzcd}
\]
The solid arrows are equivalences or canonical identifications.  The dashed arrow is the
sign-shadow comparison functor of Proposition~\ref{prop:gz-cubical-sign-shadow}; by
Proposition~\ref{prop:unit-fails-degree-minus-one}, this comparison is not an equivalence on the
full cubical category.  The arrow $\ho(\Ch_{\geq -1}(k))\to \ho(\Ch_{\geq 0}(k))$ is the
tautological degree-shift equivalence, given by reindexing chain complexes.  The arboreal arrow is
imported from \cite{Kaygun2026}; it uses the height-one shadow
$\pi\colon \OrdFor\to \Delta_{\mathrm{epi}}^{\op}$, together with the shifted identification
$\Delta_{\mathrm{epi}}^{\op}\cong \Delta_{a,\inj}$.

\section{Quillen equivalences}\label{sec:quillen-equivalences}

We now refine the Gabriel--Zisman equivalences of Section~\ref{sec:gz-localization} to Quillen
model category statements.  The chain-complex categories $\Ch_{\geq 0}(k)$ and
$\Ch_{\geq -1}(k)$ carry the standard cofibrantly generated model structures: weak equivalences
are quasi-isomorphisms, cofibrations are degreewise monomorphisms, and fibrations are degreewise
epimorphisms.

The semisimplicial and augmented semisimplicial comparisons give Quillen equivalences.  The
cubical sign-shadow functor gives a Quillen comparison with the augmented semisimplicial model
category, but Proposition~\ref{prop:unit-fails-degree-minus-one} prevents this comparison from
being a Quillen equivalence on the full cubical category.

\subsection{A transfer criterion}\label{subsec:quillen-transfer-criterion}

We first record the transfer argument in the form needed below.

\begin{lemma}\label{lem:right-induced-transfer-free}
  Let $u\colon \mathcal B\to \mathcal A$ be a $k$-linear functor, and suppose that
  $\rmod{\mathcal B}$ carries a cofibrantly generated model structure whose weak equivalences are
  detected by homology of the underlying chain complexes.  Assume that the standard generating
  trivial cofibrations may be chosen to be maps $0\to D[n]$, where each $D[n]$ is a contractible
  disk complex.  If $\mathcal A$ is free as a left $\mathcal B$-module, then the adjunction
  \[
    u_!:=-\otimes_{\mathcal B}\mathcal A \;\colon\; \rmod{\mathcal B} \;\rightleftarrows\;
    \rmod{\mathcal A} \;\colon\; u^\ast
  \]
  admits a right-induced model structure on $\rmod{\mathcal A}$.  In this model structure, a
  morphism $f$ is a weak equivalence, respectively a fibration, if and only if $u^\ast f$ is a
  weak equivalence, respectively a fibration, in $\rmod{\mathcal B}$.
\end{lemma}

\begin{proof}
  We use the right-induced model structure criterion
  \cite[Theorem~7.44]{HeutsMoerdijk2022}.  The category $\rmod{\mathcal A}$ is Grothendieck
  abelian, hence locally presentable.  Since $\mathcal A$ is free as a left $\mathcal B$-module,
  each induced disk $D[n]\otimes_{\mathcal B}\mathcal A$ is a direct sum of contractible disks.
  Thus every pushout of a map $0\to D[n]\otimes_{\mathcal B}\mathcal A$ is a split monomorphism
  with contractible cokernel.  Transfinite composites of such maps again have contractible
  cokernel, because filtered colimits are exact in $\Vect_k$.  The acyclicity hypothesis of
  \cite[Theorem~7.44]{HeutsMoerdijk2022} follows.
\end{proof}

We shall use the following elementary observation to identify the Quillen adjunctions obtained by
transfer.

\begin{lemma}\label{lem:gz-equivalence-implies-quillen-equivalence}
  Let $L\colon \mathcal M\rightleftarrows \mathcal N\colon R$ be a Quillen adjunction between
  model categories in which every object of $\mathcal N$ is fibrant.  If the right derived functor
  $\mathbf R R\colon \ho(\mathcal N)\to \ho(\mathcal M)$ agrees with an equivalence of
  Gabriel--Zisman localizations, then $L\dashv R$ is a Quillen equivalence.
\end{lemma}

\begin{proof}
  Since every object of $\mathcal N$ is fibrant, the right derived functor is represented by
  $R$ itself.  If $\mathbf R R$ is an equivalence of homotopy categories, then the derived
  adjunction has invertible unit and counit.  Hence the original Quillen adjunction is a Quillen
  equivalence.
\end{proof}

\subsection{Chain complexes and semisimplicial modules}
\label{subsec:quillen-semisimplicial}

We first treat the nonaugmented semisimplicial comparison.

\begin{theorem}\label{thm:quillen-omega-ssimp}
  The adjunction
  \begin{equation}\label{eq:quillen-ch-ssimp-nonaug}
    u_{\Delta,!}:=-\otimes_{\Omega}k[\Delta_{\inj}]
    \;\colon\;
    \Ch_{\geq 0}(k)
    \;\rightleftarrows\;
    \rmod{\Delta_{\inj}}
    \;\colon\;
    u_\Delta^\ast
  \end{equation}
  admits a right-induced model structure on $\rmod{\Delta_{\inj}}$.  In this model structure,
  $f_\bullet$ is a weak equivalence, respectively a fibration, if and only if
  $u_\Delta^\ast f_\bullet$ is a quasi-isomorphism, respectively a degreewise epimorphism.
  Moreover, \eqref{eq:quillen-ch-ssimp-nonaug} is a Quillen equivalence.
\end{theorem}

\begin{proof}
  Apply Lemma~\ref{lem:right-induced-transfer-free} to
  $u_\Delta\colon \Omega\to k[\Delta_{\inj}]$.  The required freeness is
  Lemma~\ref{lem:ssimp-free-over-omega}.  The weak equivalences in the transferred model structure
  are therefore precisely $W_{\Delta_{\inj}}$, by
  Definition~\ref{def:delta-inj-weak-equivalence}.

  The adjunction \eqref{eq:quillen-ch-ssimp-nonaug} is Quillen by construction.  Every object of
  $\rmod{\Delta_{\inj}}$ is fibrant, because fibrations are detected by degreewise epimorphisms in
  $\Ch_{\geq 0}(k)$.  Theorem~\ref{thm:gz-delta-inj-omega-equivalence} identifies the induced
  functor on Gabriel--Zisman localizations with an equivalence.  Hence
  Lemma~\ref{lem:gz-equivalence-implies-quillen-equivalence} proves that
  \eqref{eq:quillen-ch-ssimp-nonaug} is a Quillen equivalence.
\end{proof}

The augmented semisimplicial comparison is formally identical, but the right adjoint is the full
augmented chain functor.

\begin{theorem}\label{thm:quillen-omegaa-ssimp}
  The adjunction
  \begin{equation}\label{eq:quillen-ch-ssimp}
    u_{a,!}:=-\otimes_{\Omega_a}k[\Delta_{a,\inj}]
    \;\colon\;
    \Ch_{\geq -1}(k)
    \;\rightleftarrows\;
    \rmod{\Delta_{a,\inj}}
    \;\colon\;
    u_a^\ast=C_\ast^a
  \end{equation}
  admits a right-induced model structure on $\rmod{\Delta_{a,\inj}}$.  In this model structure,
  $f_\bullet$ is a weak equivalence, respectively a fibration, if and only if $C_\ast^a(f_\bullet)$
  is a quasi-isomorphism, respectively a degreewise epimorphism.  Moreover,
  \eqref{eq:quillen-ch-ssimp} is a Quillen equivalence.
\end{theorem}

\begin{proof}
  Apply Lemma~\ref{lem:right-induced-transfer-free} to
  $u_a\colon \Omega_a\to k[\Delta_{a,\inj}]$.  The required freeness is
  Lemma~\ref{lem:augmented-ssimp-free-over-omega-a}.  The transferred weak equivalences are
  precisely $W_{\Delta_{a,\inj}}$, by Theorem~\ref{thm:delta-a-inj-weak-equivalence}.

  The adjunction \eqref{eq:quillen-ch-ssimp} is Quillen by construction, and every object of
  $\rmod{\Delta_{a,\inj}}$ is fibrant.  Theorem~\ref{thm:gz-ssimp-chain-equivalence} identifies the
  induced functor on Gabriel--Zisman localizations with an equivalence.  Therefore
  Lemma~\ref{lem:gz-equivalence-implies-quillen-equivalence} proves that
  \eqref{eq:quillen-ch-ssimp} is a Quillen equivalence.
\end{proof}

\begin{remark}\label{rem:quillen-shift}
  The categories $\Ch_{\geq 0}(k)$ and $\Ch_{\geq -1}(k)$ are, of course, isomorphic after
  relabelling degrees.  This reindexing is not one of the comparison functors constructed from the
  indexing categories, and we do not use it to identify the semisimplicial and augmented
  semisimplicial Quillen adjunctions.
\end{remark}

\subsection{The cubical sign-shadow comparison}\label{subsec:quillen-cubical}

We now transfer the augmented semisimplicial model structure along the sign-shadow restriction
functor $v^\ast\colon \rmod{\square_{\inj}}\to \rmod{\Delta_{a,\inj}}$.

\begin{theorem}\label{thm:quillen-cubical-comparison}
  The adjunction
  \begin{equation}\label{eq:quillen-cubical-comparison}
    v_!:=-\otimes_{k[\Delta_{a,\inj}]}k[\square_{\inj}]
    \;\colon\;
    \rmod{\Delta_{a,\inj}}
    \;\rightleftarrows\;
    \rmod{\square_{\inj}}
    \;\colon\;
    v^\ast
  \end{equation}
  admits a right-induced model structure on $\rmod{\square_{\inj}}$.  In this model structure,
  $f_\bullet$ is a weak equivalence, respectively a fibration, if and only if $v^\ast f_\bullet$ is
  a weak equivalence, respectively a fibration, in $\rmod{\Delta_{a,\inj}}$.  These weak
  equivalences are exactly the cubical weak equivalences, and \eqref{eq:quillen-cubical-comparison}
  is a Quillen adjunction.
\end{theorem}

\begin{proof}
  Apply Lemma~\ref{lem:right-induced-transfer-free} to
  $v\colon k[\Delta_{a,\inj}]\to k[\square_{\inj}]$, using the right-induced model structure on
  $\rmod{\Delta_{a,\inj}}$ from Theorem~\ref{thm:quillen-omegaa-ssimp}.  The required freeness is
  Lemma~\ref{lem:scube-free-over-ssimp}.  Thus the right-induced model structure on
  $\rmod{\square_{\inj}}$ exists, and \eqref{eq:quillen-cubical-comparison} is Quillen by
  construction.

  It remains only to identify the weak equivalences.  A morphism $f_\bullet$ in
  $\rmod{\square_{\inj}}$ is a weak equivalence in the transferred model structure exactly when
  $v^\ast f_\bullet\in W_{\Delta_{a,\inj}}$.  By Proposition~\ref{prop:gz-cubical-sign-shadow},
  this condition is equivalent to $f_\bullet\in W_{\square_{\inj}}$.  Hence the transferred weak
  equivalences are precisely the cubical weak equivalences.
\end{proof}

\begin{remark}\label{rem:quillen-cubical-not-equivalence}
  The adjunction \eqref{eq:quillen-cubical-comparison} is not a Quillen equivalence on the full
  cubical category.  Indeed, a Quillen equivalence induces an equivalence on homotopy categories,
  while Proposition~\ref{prop:unit-fails-degree-minus-one} shows that the unit of the sign-shadow
  adjunction fails to be an isomorphism after Gabriel--Zisman localization.
\end{remark}

\subsection{The Quillen comparison diagram}\label{subsec:quillen-comparison-diagram}

The Quillen-level comparisons established in this section, together with the arboreal comparison
imported from \cite{Kaygun2026}, are organized as follows:
\[
\begin{tikzcd}[column sep=6.2em, row sep=4.2em]
&
\rmod{\OrdFor}
  \arrow[d, "{\simeq\ \text{\cite{Kaygun2026}}}"]
&
\\
\rmod{\square_{\inj}}
  \arrow[r, shift left=1.1ex, dashed,
    "{v^\ast\ \text{(Thm.~\ref{thm:quillen-cubical-comparison})}}"]
&
\rmod{\Delta_{a,\inj}}
  \arrow[l, shift left=1.1ex, dashed,
    "{v_!\ \text{(Thm.~\ref{thm:quillen-cubical-comparison})}}"]
  \arrow[r, shift left=1.1ex,
    "{u_a^\ast=C_\ast^a\ \text{(Thm.~\ref{thm:quillen-omegaa-ssimp})}}"]
&
\Ch_{\geq -1}(k)
  \arrow[d, "{\simeq\ \text{degree shift}}"]
  \arrow[l, shift left=1.1ex,
    "{u_{a,!}\ \text{(Thm.~\ref{thm:quillen-omegaa-ssimp})}}"]
\\
&
\rmod{\Delta_{\inj}}
  \arrow[r, shift left=1.1ex,
    "{u_\Delta^\ast\ \text{(Thm.~\ref{thm:quillen-omega-ssimp})}}"]
&
\Ch_{\geq 0}(k)
  \arrow[l, shift left=1.1ex,
    "{u_{\Delta,!}\ \text{(Thm.~\ref{thm:quillen-omega-ssimp})}}"]
\end{tikzcd}
\]
The two solid horizontal adjunctions are Quillen equivalences. The dashed horizontal adjunction is
Quillen but not a Quillen equivalence, by
Remark~\ref{rem:quillen-cubical-not-equivalence}. The left vertical arrow denotes the arboreal
Quillen equivalence imported from \cite{Kaygun2026}. The right vertical arrow is the tautological
Quillen equivalence obtained by reindexing degrees: $\Ch_{\geq 0}(k)$ and $\Ch_{\geq -1}(k)$
are canonically isomorphic as model categories after shifting the homological indexing by one.

\subsection*{Acknowledgements}

The author acknowledges the help of large language models for copy-editing the manuscript.

\bibliographystyle{amsalpha}
\bibliography{references}

\end{document}